\documentclass[11pt]{article}
\usepackage{enumerate}
\usepackage{amssymb,a4wide,latexsym,makeidx,epsfig,fleqn}
\usepackage{amsthm}
\usepackage{amsmath}
\usepackage{enumerate}
\usepackage{graphicx}
\usepackage{float}
\usepackage{tikz}
\usepackage[colorlinks=true, linkcolor=blue, citecolor=red, urlcolor=blue]{hyperref}
\allowdisplaybreaks[4]
\newtheorem{theorem}{Theorem}[section]

\newtheorem{lemma}[theorem]{Lemma}

\newtheorem{problem}{Problem}

\begin{document}
\textwidth 150mm \textheight 225mm
\title{Distance spectral radius conditions for edge-disjoint spanning trees and a forest with constraints\thanks{Supported by the National Natural Science Foundation of China (No. 12271439).}}
\author{{Yongbin Gao$^{a,b}$, Ligong Wang$^{a,b,}$\footnote{Corresponding author.}}\\
{\small $^a$ School of Mathematics and Statistics, Northwestern
Polytechnical University,}\\ {\small  Xi'an, Shaanxi 710129,
P.R. China.}\\
{\small $^b$ Xi'an-Budapest Joint Research Center for Combinatorics, Northwestern
Polytechnical University,}\\
{\small Xi'an, Shaanxi 710129,
P.R. China. }\\
{\small E-mail: gybmath@163.com, lgwangmath@163.com} }
\date{}
\maketitle
\begin{center}
\begin{minipage}{120mm}
\vskip 0.3cm
\begin{center}
{\small {\bf Abstract}}
\end{center}
{\small Let $k\ge 2$ be a positive integer and let $G$ be a simple graph of order $n$ with minimum degree $\delta$. A graph $G$ is said to have property $P(k, d)$ if it contains $k$ edge-disjoint spanning trees and an additional forest $F$ with edge number $|E(F)| > \frac{d-1}{d}(n-1)$, such that if $F$ is not a spanning tree, then $F$ has a component with at least $d$ edges. Let $D(G)$ be the distance matrix of $G$. We denote $\rho_D(G)$ as the largest eigenvalue of $D(G)$, which is called the distance spectral radius of $G$.
In this paper, we investigate the relationship between the distance spectral radius and the property $P(k, \delta)$. We prove that for a connected graph $G$ of order $n \ge 2k+8$ with minimum degree $\delta \ge k+2$, if $\rho_D(G) \le \rho_D(K_{k-1} \vee (K_{n-k} \cup K_1))$, then $G$ possesses property $P(k, \delta)$. Furthermore, for a connected balanced bipartite graph $G$ of order $n \ge 4k+8$ with minimum degree $\delta \ge k+2$, we show that if $\rho_D(G) \le \rho_D(K_{\frac{n}{2}, \frac{n}{2}} \setminus E(K_{1, \frac{n}{2}-k+1}))$, then $G$ also possesses property $P(k, \delta)$. Our results generalize the work of Fan et al. [Discrete Appl. Math. 376 (2025), 31–40] from the existence of $k$ edge-disjoint spanning trees to the more refined structural property $P(k, \delta)$.

 }
\end{minipage}
\end{center}

\section{Introduction}

In this paper, we consider finite, undirected, and simple graphs. Let $G$ be a graph with vertex set $V(G) = \{v_1, v_2,\ldots, v_n\}$ and edge set $E(G)$. We denote the order of $G$ by $n = |V(G)|$. For any vertex $v_i \in V(G)$($1 \le i \le n$), let $d_G(v_i)$ (or simply $d_i$) denote the degree of $v_i$, and let $\delta(G)$ (or simply $\delta$) be the minimum degree of $G$. For any two disjoint subsets $X, Y \subseteq V(G)$, let $E(X, Y)$ be the set of edges connecting a vertex in $X$ to a vertex in $Y$, and denote $e(X, Y) = |E(X, Y)|$. The spanning tree packing number of a graph $G$, denoted by $\tau(G)$, is defined as the maximum number of edge-disjoint spanning trees contained in $G$. For any two vertices $v_i, v_j \in V(G)$($1 \le i,j \le n$) of a connected graph $G$, the distance $d_G(v_i, v_j)$ is the length of a shortest path connecting them. The distance matrix of $G$, denoted by $D(G)$, is the $n \times n$ matrix whose $(v_i, v_j)$-entry is $d_G(v_i, v_j)$(or $d_{ij}$). The largest eigenvalue of $D(G)$, denoted by $\rho_D(G)$, is called the distance spectral radius of $G$.

The study of the relationship between graph spectra and spanning trees traces back to the Matrix-Tree Theorem proposed by Kirchhoff \cite{Kirchhoff1847} while investigating electrical networks. This work revealed the fundamental link between the total number of spanning trees and the Laplacian eigenvalues of a graph. For edge-disjoint spanning trees, Seymour proposed the following problem in a private communication with Cioab\v{a} (see \cite{Cioaba2012}).

\begin{problem}\label{prob1}
	Let $G$ be a connected graph. Determine the relationship between the spanning tree packing number $\tau(G)$ and the eigenvalues of $G$.
\end{problem}

Let $A(G)$ be the adjacency matrix of a graph $G$ of order $n$. The eigenvalues of $A(G)$ are denoted by $\lambda_1(G) \ge \lambda_2(G) \ge \cdots \ge \lambda_n(G)$. Motivated by Problem \ref{prob1}, Cioab\v{a} and Wong \cite{Cioaba2012} investigated the relationship between $\tau(G)$ and $\lambda_2(G)$. They established sufficient spectral conditions for $d$-regular graphs to satisfy $\tau(G) \ge k$ for $k=2,3$, and proposed a general conjecture for $4 \le k \le \lfloor d/2 \rfloor$. Subsequently, Gu et al. \cite{Gu2016} generalized this conjecture to graphs with minimum degree $\delta$. This conjecture was fully resolved by Liu et al. \cite{Liu2014, 2Liu2014}. They proved that for a graph $G$ with minimum degree $\delta \ge 2k$, if $\lambda_2(G) < \delta - \frac{2k-1}{\delta+1}$, then $\tau(G) \ge k$. The tightness of the bound was later confirmed by Cioab\v{a} et al. \cite{Cioaba2022} through the construction of extremal graphs. Furthermore, Liu et al. \cite{Liu2019} studied the spectral conditions for graphs with given girth. Beyond $\lambda_2(G)$, spectral conditions involving other eigenvalues, such as $\lambda_1(G)$, $\lambda_3(G)$ and Laplacian eigenvalues, have also been studied (see, e.g., \cite{Duan2020, Fan2023, Hong2016, Hu2023}).

Recently, Fan et al. \cite{Fan2025} established the relationship between the distance spectral radius and $\tau(G)$ by constructing extremal graphs. Let $D(G)$ be the distance matrix of $G$, and let $\rho_D(G)$ denote the distance spectral radius (the largest eigenvalue of $D(G)$). We denote by $G_1 \vee G_2$ the join of two disjoint graphs $G_1$ and $G_2$, which is obtained from the union $G_1 \cup G_2$ by adding all edges between $V(G_1)$ and $V(G_2)$. For a graph $H$ and a subset of edges $E' \subseteq E(H)$, let $H \setminus E'$ denote the graph obtained from $H$ by deleting the edges in $E'$. Let $K_{n,n}$ denote the complete bipartite graph with $n$ vertices in each part. A bipartite graph with partition sets of equal size is called balanced. Fan et al. \cite{Fan2025} obtained the following results.

\begin{theorem}[\cite{Fan2025}] \label{thm:Fan1}
	Let $k \ge 2$ be an integer and $G$ be a connected graph of order $n \ge 2k+6$. If
	\[
	\rho_D(G) \le \rho_D(K_{k-1} \vee (K_{n-k} \cup K_1)),
	\]
	then $\tau(G) \ge k$, unless $G \cong K_{k-1} \vee (K_{n-k} \cup K_1)$.
\end{theorem}

\begin{theorem}[\cite{Fan2025}] \label{thm:Fan2}
	Let $k \ge 2$ be an integer and $G$ be a connected balanced bipartite graph of order $n \ge 4k+4$. If
	\[
	\rho_D(G) \le \rho_D\left(K_{\frac{n}{2}, \frac{n}{2}} \setminus E\left(K_{1, \frac{n}{2}-k+1}\right)\right),
	\]
	then $\tau(G) \ge k$, unless $G \cong K_{\frac{n}{2}, \frac{n}{2}} \setminus E\left(K_{1, \frac{n}{2}-k+1}\right)$.
\end{theorem}

A fundamental tool for studying edge-disjoint spanning trees is the classical Tree Packing Theorem, established by Nash-Williams \cite{Nash-Williams1961} and Tutte \cite{Tutte1961}. Let $\mathcal{P} = \{V_1, V_2, \ldots, V_p\}$ be a partition of the vertex set $V(G)$ into $p$ parts, where $p = |\mathcal{P}| \ge 2$. For any nontrivial graph $G$, the fractional packing number $\nu_f(G)$ is defined by
\[
\nu_f(G) = \min_{p \ge 2} \frac{\sum_{1 \le i < j \le p} e(V_i, V_j)}{p-1}.
\]
\begin{theorem}[\cite{Nash-Williams1961, Tutte1961}] \label{thm:TreePacking}
	For a graph $G$ and an integer $k \ge 1$, $\tau(G) \ge k$ if and only if $\nu_f(G) \ge k$.
\end{theorem}

Spectral conditions related to the fractional packing number were first investigated by Hong et al. \cite{Hong2016}. More recently, Fang and Yang \cite{Fang2025} offered a structural interpretation for the fractional part of $\nu_f(G)$. Based on this result, Cai and Zhou \cite{Cai2026} formally introduced the property $P(k, d)$. A graph $G$ possesses property $P(k, d)$ if the following conditions hold:
\begin{enumerate}[{\rm(a)}]
	\item $G$ contains $k$ edge-disjoint spanning trees (i.e., $\tau(G) \ge k$);
	\item apart from these $k$ trees, there exists a forest $F$ such that $|E(F)| > \frac{d-1}{d}(n-1)$;
	\item if $F$ is not a spanning tree, then $F$ has a component with at least $d$ edges.
\end{enumerate}

The following theorem by Fang and Yang \cite{Fang2025} (see also \cite{Cai2026}) connects the fractional packing number to property $P(k, d)$ and is essential for our proofs.

\begin{theorem}[\cite{Cai2026, Fang2025}] \label{thm:fractional packing}
	Let $k$ and $d$ be positive integers. For a nontrivial graph $G$, if 
	\[\nu_f(G) > k+\frac{d-1}{d},\] 
	then $G$ has property $P(k, d)$.
\end{theorem}

Naturally, this raises a fundamental question.

\begin{problem}\label{prob2}
	Let $k$ be a positive integer, and let $G$ be a nontrivial graph with minimum degree $\delta$. What conditions on the eigenvalues of $G$ are sufficient to guarantee that $G$ has property $P(k, \delta)$?
\end{problem}

Cai and Zhou \cite{Cai2026} initiated the study of Problem \ref{prob2} by establishing conditions involving the spectral radius $\rho(G)$ and the second largest eigenvalue $\lambda_2(G)$. In this paper, we extend this line of research to the distance matrix of a graph.
Inspired by the work of Cai and Zhou \cite{Cai2026} and Fan et al. \cite{Fan2025}, this paper extends Theorem \ref{thm:Fan1} and Theorem \ref{thm:Fan2} in \cite{Fan2025} to the property $P(k, \delta)$. Our main results are stated as follows.

\begin{theorem} \label{thm:Main1}
	Let $k \ge 2$ be an integer and $G$ be a connected graph of order $n \ge 2k+8$ with minimum degree $\delta \ge k+2$. If
	\[
	\rho_D(G) \le \rho_D(K_{k-1} \vee (K_{n-k} \cup K_1)),
	\]
	then $G$ has property $P(k, \delta)$.
\end{theorem}

\begin{theorem} \label{thm:Main2}
	Let $k \ge 2$ be an integer and $G$ be a connected balanced bipartite graph of order $n \ge 4k+8$ with minimum degree $\delta \ge k+2$. If
	\[
	\rho_D(G) \le \rho_D\left(K_{\frac{n}{2}, \frac{n}{2}} \setminus E\left(K_{1, \frac{n}{2}-k+1}\right)\right),
	\]
	then $G$ has property $P(k, \delta)$.
\end{theorem}
 In Section 2, we present some lemmas that will be used in our proofs. The proofs of Theorem \ref{thm:Main1} and Theorem \ref{thm:Main2} are given in Section 3.

\section{Preliminaries}

In this section, we present some necessary lemmas that will be used in our proofs. 

The following lemma provides an upper bound for the distance spectral radius of the extremal graph $K_{k-1} \vee (K_{n-k} \cup K_1)$. This result was proved by Fan et al. \cite{Fan2025} under the condition $n \ge 2k+6$. Since our main theorem requires $n \ge 2k+8$, this bound naturally applies.

\begin{lemma}[\cite{Fan2025}] \label{lem:Fan1_bound}
	Let $k \ge 2$ be an integer. If $n \ge 2k+6$, then
	\[
	\rho_D(K_{k-1} \vee (K_{n-k} \cup K_1)) < n+2.
	\]
\end{lemma}

The Wiener index of a graph $G$, denoted by $W(G)$, is defined as the sum of distances between all unordered pairs of vertices, i.e., $W(G) = \sum_{1\le i < j \le n} d_{ij}$. The following lemma gives a lower bound for the distance spectral radius in terms of the Wiener index.

\begin{lemma} \label{lem:Wiener}
	Let $G$ be a connected graph of order $n$. Then
	\[
	\rho_D(G) \ge \frac{2W(G)}{n}.
	\]
\end{lemma}

\begin{proof}
	Let $\mathbf{1}$ be the all-ones vector of dimension $n$. By the Rayleigh quotient principle, we have
	\[
	\rho_D(G) = \max_{x \in \mathbb{R}^n, x \ne 0} \frac{x^T D(G) x}{x^T x} \ge \frac{\mathbf{1}^T D(G) \mathbf{1}}{\mathbf{1}^T \mathbf{1}} = \frac{\sum_{1\le i , j \le n} d_{ij}}{n} = \frac{2W(G)}{n}.
	\]
\end{proof}

\begin{lemma}[\cite{Fan2025}] \label{lem:comb_ineq}
	Let $a$ and $b$ be two positive integers. If $a \ge b$, then
	\[
	\binom{a}{2} + \binom{b}{2} < \binom{a+1}{2} + \binom{b-1}{2}.
	\]
\end{lemma}

The next lemma is an algebraic inequality that plays a crucial role in the proof for the bipartite case. The proof is purely algebraic and identical to that of Lemma 3.4 in \cite{Fan2025}, so we omit it here.

\begin{lemma}[\cite{Fan2025}] \label{lem:alg_ineq}
	Let $s \ge 1$ be an integer. Let $a_1, \dots, a_s$ and $b_1, \dots, b_s$ be non-negative integers such that $a_i+b_i \ge 2$ for all $1 \le i \le s$. Let $a = \sum_{i=1}^s a_i$ and $b = \sum_{i=1}^s b_i$. If $b \ge a$, then
	\[
	\sum_{i=1}^s a_i b_i \le a(b-(s-1)).
	\]
\end{lemma}

Finally, we state the bound for the distance spectral radius of the balanced bipartite extremal graph, derived in \cite{Fan2025}. Note that in our context, we require $n \ge 4k+8$, which is stronger than the condition $n \ge 4k+4$ in the original lemma.

\begin{lemma}[\cite{Fan2025}] \label{lem:Fan2_bound}
	Let $k \ge 2$ be an integer. If $n \ge 4k+4$, then
	\[
	\rho_D\left(K_{\frac{n}{2}, \frac{n}{2}} \setminus E\left(K_{1, \frac{n}{2}-k+1}\right)\right) < \frac{3n}{2} + 1.
	\]
\end{lemma}

\section{Proofs of Main Results}

In this section, we present the proofs of Theorem \ref{thm:Main1} and Theorem \ref{thm:Main2}.
\begin{proof}[Proof of Theorem \ref{thm:Main1}]
	Suppose to the contrary that $G$ is a graph of order $n\ge 2k+8$ and it does not have property $P(k, \delta)$. By Theorem \ref{thm:fractional packing}, this implies that the fractional packing number satisfies
	\[
	\nu_f(G)\le k+\frac{\delta-1}{\delta}.
	\]
	According to the definition of $\nu_f(G)$, there exists a partition $\mathcal{P} = \{V_1, V_2, \dots, V_s\}$ of $V(G)$ with $s \ge 2$ such that
	\begin{equation} \label{eq:1}
		\sum_{1 \le i < j \le s} e(V_i, V_j)\le (k + \frac{\delta-1}{\delta})(s-1) < (k+1)(s-1).
	\end{equation}
	By the assumption and Lemma \ref{lem:Fan1_bound}, we have 
	\begin{equation} \label{eq:2}
		\rho_D(G) \le \rho_D(K_{k-1} \vee (K_{n-k} \cup K_1)) < n+2.
	\end{equation}
	On the other hand, let $d_G(v_i)$ (or simply $d_i$) denote the degree of vertex $v_i$. Then a lower bound for the Wiener index $W(G)$ is given by
	\[
	W(G) \ge \frac{1}{2}\sum_{i=1}^{n}(d_i + 2(n-1-d_i)) = n(n-1) - e(G).
	\]
	Using Lemma \ref{lem:Wiener}, we obtain a lower bound for the distance spectral radius of $G$
	\begin{equation} \label{eq:3}
		\rho_D(G) \ge \frac{2W(G)}{n} \ge 2(n-1) - \frac{2e(G)}{n}.
	\end{equation}
	Combining Inequalities \eqref{eq:2} and \eqref{eq:3} yields a lower bound for the total number of edges $e(G)$, namely
	\begin{equation} \label{eq:4}
		e(G) > \frac{n(n-4)}{2}.
	\end{equation}
	
	Let $|V_i|=n_i$ for $1 \le i \le s$. Without loss of generality, assume that $n_1 \le n_2 \le \dots \le n_s$. We now consider the possible values of $s$.
	
	\noindent\textbf{Case 1. $s \ge 4$.}
	
	Combining Inequality \eqref{eq:1} and Lemma \ref{lem:comb_ineq}, we obtain:
	\begin{align*}
		e(G) &= \sum_{i=1}^{s} e(G[V_i]) + \sum_{1 \le i < j \le s} e(V_i, V_j) \\
		&< \sum_{i=1}^{s} \binom{n_i}{2} + (k+1)(s-1) \\
		&\le \binom{n-s+1}{2} +(k+1)(s-1) \\
		&\le \frac{(n-s+1)(n-s)}{2} + \frac{(n-6)(s-1)}{2} && (\text{since } n \ge 2k+8) \\
		&= \frac{n^2-ns+s^2-7s+6}{2} \\
		&= \frac{n(n-4)}{2} + \left(\frac{(4-s)n}{2} + \frac{s^2}{2} - \frac{7s}{2} + 3\right) \\
		&\le \frac{n(n-4)}{2} + \left(-\frac{3s}{2} + 3\right) && (\text{since } n \ge s)\\
		&< \frac{n(n-4)}{2}. && (\text{since } s \ge 4)
	\end{align*}
	This contradicts Inequality \eqref{eq:4}. Thus, $s \le 3$.

	\noindent\textbf{Case 2. $s = 2$.}

	From Inequality \eqref{eq:1}, we have $e(V_1, V_2) \le k + \frac{\delta-1}{\delta}$.
	Since $e(V_1, V_2)$ is an integer and $0 < \frac{\delta-1}{\delta} < 1$, we have $e(V_1, V_2) \le \lfloor k + \frac{\delta-1}{\delta} \rfloor = k$.
	
	We next assert that $n_1=1$. If not, $n_1 \ge 2$. Combining this with Lemma \ref{lem:comb_ineq}, we get
	\begin{align*}
		e(G) &= e(G[V_1]) + e(G[V_2]) + e(V_1, V_2) \\
		&\le \binom{n_1}{2} + \binom{n_2}{2} + k \\
		&\le \binom{2}{2} + \binom{n-2}{2}+k && (\text{since } n_1 \ge 2)\\
		&= \frac{n^2-5n+6}{2}+k+1 \\
		&= \frac{n(n-4)}{2} - \left(\frac{n}{2} - k - 4\right) \\
		&\le \frac{n(n-4)}{2}, && (\text{since } n \ge 2k+8)
	\end{align*}
	which contradicts Inequality \eqref{eq:4}. It follows that $n_1=1$.

	Without loss of generality, we assume $V_1=\{v_1\}$. Noting that $d_G(v_1)\ge \delta$ and
	\[
	d_G(v_1)= e(V_1, V_2)\le k+\frac{\delta-1}{\delta}=k+1-\frac{1}{\delta},
	\]
	we obtain
	\[
	\delta\le k+1-\frac{1}{\delta} < k+1,
	\]
	which contradicts the assumption that $\delta\ge k+2$.
	
	\noindent\textbf{Case 3. $s = 3$.}
	
	From Inequality \eqref{eq:1}, we know that $\sum_{1 \le i < j \le 3} e(V_i, V_j) \le 2(k+\frac{\delta-1}{\delta}) < 2(k+1)$, which implies $\sum_{1 \le i < j \le 3} e(V_i, V_j) \le 2k+1$. 
	
	We assert that $n_1=n_2=1$. Suppose to the contrary that this is not true. Assume $n_2 \ge 2$ and $n_1 \ge 1$. By Lemma \ref{lem:comb_ineq}, we have
	\begin{align*}
		e(G) &= e(G[V_1]) + e(G[V_2]) + e(G[V_3]) + \sum_{1 \le i < j \le 3} e(V_i, V_j) \\ 
		&\le \binom{n_1}{2} + \binom{n_2}{2} + \binom{n_3}{2} + 2k+1 \\
		&\le \binom{2}{2} + \binom{n-3}{2} + 2k+1 \\
		&= 1 + \frac{(n-3)(n-4)}{2} + 2k+1 \\
		&= \frac{n(n-4)}{2} + 2k+8-\frac{3}{2}n \\
		&< \frac{n(n-4)}{2} && (\text{since } n \ge 2k+8)
	\end{align*}
	This contradicts Inequality \eqref{eq:4}. Thus, we must have $n_1 = n_2 = 1$.
	
	Without loss of generality, we assume $V_1=\{v_1\}$ and $V_2=\{v_2\}$. Noting that $d_G(v_1), d_G(v_2)\ge \delta$ and 
	\[
	d_G(v_1) + d_G(v_2)= \sum_{1 \le i < j \le 3} e(V_i, V_j) + e(V_1, V_2)\le 2(k+\frac{\delta-1}{\delta})+1=2k+3 - \frac{2}{\delta},
	\]
	we obtain
	\[
	2\delta\le 2k+3-\frac{2}{\delta} < 2k+3,
	\]
	which contradicts the assumption that $\delta\ge k+2$.
	
	All cases lead to a contradiction. This completes the proof.
\end{proof}
	We are now ready to prove Theorem \ref{thm:Main2}.
	
\begin{proof}[Proof of Theorem \ref{thm:Main2}]
	Suppose to the contrary that $G = (X, Y)$ is a balanced bipartite graph of order $n\ge 4k+8$ and it does not have property $P(k, \delta)$. By Theorem \ref{thm:fractional packing}, this implies that the fractional packing number satisfies
	\[
	\nu_f(G)\le k+\frac{\delta-1}{\delta}.
	\]
	According to the definition of $\nu_f(G)$, there exists a partition $\mathcal{P} = \{V_1, V_2, \dots, V_s\}$ of $V(G)$ with $s \ge 2$ such that
	\begin{equation} \label{eq:5}
		\sum_{1 \le i < j \le s} e(V_i, V_j)\le (k + \frac{\delta-1}{\delta})(s-1) < (k+1)(s-1).
	\end{equation}
	By the assumption and Lemma \ref{lem:Fan2_bound}, we have 
	\begin{equation} \label{eq:6}
		\rho_D(G) \le \rho_D\left(K_{\frac{n}{2}, \frac{n}{2}} \setminus E\left(K_{1, \frac{n}{2}-k+1}\right)\right) < \frac{3n}{2} + 1.
	\end{equation}
	On the other hand, let $d_G(v_i)$ (or simply $d_i$) denote the degree of vertex $v_i$. Then a lower bound for the Wiener index $W(G)$ is given by
	\[
	W(G) \ge \frac{1}{2}\sum_{i=1}^{n}(d_i + 2(\frac{n}{2}-1) + 3(\frac{n}{2}-d_i)) = \frac{1}{2}(\frac{5n}{2}-2)n - 2e(G).
	\]
	Using Lemma \ref{lem:Wiener}, we obtain a lower bound for the distance spectral radius
	\begin{equation} \label{eq:7}
		\rho_D(G) \ge \frac{2W(G)}{n}\ge \frac{5n}{2}-2 - \frac{4e(G)}{n}.
	\end{equation}
	Combining Inequalities \eqref{eq:6} and \eqref{eq:7} yields a lower bound for the total number of edges $e(G)$, namely
	\begin{equation} \label{eq:8}
		e(G) > \frac{n(n-3)}{4}.
	\end{equation}
	
	Clearly, $n \ge s$. We assert that $n \ge s+1$. Suppose to the contrary that $n=s$. Then by the condition $n \ge 4k+8$ and Inequality \eqref{eq:5}, we have
	\[
	e(G) = \sum_{1 \le i < j \le s} e(V_i, V_j) < (k+1)(n-1) \le \frac{(n-4)(n-1)}{4} < \frac{n(n-3)}{4},
	\]
	which contradicts Inequality \eqref{eq:8}. Thus, $n \ge s+1$.
	
	Without loss of generality, assume that $|V_1| \le |V_2| \le \dots \le |V_s|$. Let $|V_i \cap X| = a_i$ and $|V_i \cap Y| = b_i$ for $1 \le i \le s$.
	
	\noindent\textbf{Case 1. $s \ge 4$.}
	
	Suppose that there are $s_1$ trivial parts in $X$ (referring to the parts in the partition $\mathcal{P} = \{V_1, V_2, \dots, V_s\}$) and $s_2$ trivial parts in $Y$. Without loss of generality, assume $s_1 \ge s_2$. Then we have $\frac{n}{2} - s_1 \le \frac{n}{2} - s_2$. Let $s' = s - s_1 - s_2$. By Lemma \ref{lem:alg_ineq} and Inequality \eqref{eq:1}, we obtain
	\begin{align*}
		e(G) &= \sum_{i=1}^{s} e(G[V_i]) + \sum_{1 \le i < j \le s} e(V_i, V_j) \\
		&< \sum_{i=1}^{s} a_ib_i + (k+1)(s-1) \\
		&\le \left(\frac{n}{2} - s_1\right)\left(\frac{n}{2}-s_2-(s'-1)\right) + (k+1)(s-1) \\
		&= \frac{n^2}{4} - \frac{n}{2}(s_1+s_2+s'-1) + s_1(s_2+s'-1) + (k+1)(s-1) \\
		&\le \frac{n^2}{4} - \frac{n}{2}(s_1+s_2+s'-1) + \left(\frac{s_1+s_2+s'-1}{2}\right)^2 + (k+1)(s-1) && (\text{since } xy \le ((x+y)/2)^2) \\
		&= \frac{n^2}{4} - \frac{n}{2}(s-1) + \frac{(s-1)^2}{4} + (k+1)(s-1) \\
		&= \frac{n(n-3)}{4} - \left(\frac{n(2s-5)}{4} - \frac{(s-1)^2}{4} - k(s-1) - s + 1\right) \\
		&\le \frac{n(n-3)}{4} - \left(\frac{n(2s-5)}{4} - \frac{(s-1)^2}{4} - \frac{n-8}{4}(s-1) - s+1\right) && (\text{since } n \ge 4k+8) \\
		&= \frac{n(n-3)}{4} - \frac{(s-4)n-s^2+6s-5}{4} \\
		&\le \frac{n(n-3)}{4} - \frac{(s-4)(s+1)-s^2+6s-5}{4} && (\text{since } n \ge s+1 \text{ and } s \ge 4) \\
		&= \frac{n(n-3)}{4} - \frac{3s-9}{4} \\
		&< \frac{n(n-3)}{4}. && (\text{since } s \ge 4)
	\end{align*}
	This contradiction implies that $s \le 3$.
	
	\noindent\textbf{Case 2. $s = 2$.}
	
	From Inequality \eqref{eq:5}, we have $e(V_1, V_2) \le k + \frac{\delta-1}{\delta}$.
	Since $e(V_1, V_2)$ is an integer and $0 < \frac{\delta-1}{\delta} < 1$, we obtain $e(V_1, V_2) \le \lfloor k + \frac{\delta-1}{\delta} \rfloor = k$.
	
	We assert that $|V_1| = 1$. Suppose to the contrary that $|V_1| \ge 2$. Without loss of generality, assume $a_1 \ge b_1$. If $b_1 = 0$, then $a_1 \ge 2$ since $|V_1| = a_1 + b_1 \ge 2$.
	Let $V_1 = \{v_1, v_2, \ldots, v_{a_1}\}$. Since $G$ is bipartite and $V_1 \subseteq X$, there are no edges within $V_1$. Using $\delta \ge k+2$, the sum of degrees for vertices in $V_1$ satisfies $2(k+2)\le \sum_{i=1}^{a_1}d_i = e(V_1,V_2)\le k$, which is a contradiction clearly.
	Thus, we must have $b_1 \ge 1$.
	
	Since $a_1 + a_2 = b_1 + b_2 = \frac{n}{2}$, we have
	\begin{align*}
		e(G) &= e(G[V_1]) + e(G[V_2]) + e(V_1, V_2) \\
		&\le a_1b_1 + a_2b_2 + k \\
		&= a_1b_1 + \left(\frac{n}{2} - a_1\right)\left(\frac{n}{2} - b_1\right) + k \\
		&\le a_1b_1 + \left(\frac{n}{2} - a_1\right)\left(\frac{n}{2} - b_1\right) + \frac{n-8}{4} && (\text{since } n \ge 4k+8) \\
		&= \frac{n(n-3)}{4} - \left(\frac{(a_1+b_1-2)n}{2} - 2a_1b_1 + 2\right) \\
		&\le \frac{n(n-3)}{4} - \left((a_1+b_1-2)(a_1+b_1) - 2a_1b_1 + 2\right) && (\text{since } n = |V_1| + |V_2| \ge 2(a_1+b_1)) \\
		&\le \frac{n(n-3)}{4} - \left(4(\sqrt{a_1 b_1}-1)\sqrt{a_1b_1} - 2a_1b_1 + 2\right) && (\text{since } a_1+b_1 \ge 2\sqrt{a_1 b_1}) \\
		&= \frac{n(n-3)}{4} - 2(\sqrt{a_1b_1} - 1)^2 \\
		&\le \frac{n(n-3)}{4}.
	\end{align*}
	This contradicts Inequality \eqref{eq:8}. Thus, we must have $|V_1| = 1$.
	
	Without loss of generality, we assume $V_1=\{v_1\}$. Noting that $d_G(v_1)\ge \delta$ and
	\[
	d_G(v_1)= e(V_1, V_2)\le k+\frac{\delta-1}{\delta}=k+1-\frac{1}{\delta},
	\]
	we have
	\[
	\delta\le k+1-\frac{1}{\delta} < k+1,
	\]
	which contradicts the assumption that $\delta\ge k+2$.
	
	\noindent\textbf{Case 3. $s = 3$.}
	
	From Inequality \eqref{eq:1}, we know that $\sum_{1 \le i < j \le 3} e(V_i, V_j) \le 2(k+\frac{\delta-1}{\delta}) < 2(k+1)$, which implies $\sum_{1 \le i < j \le 3} e(V_i, V_j) \le 2k+1$.
	
	Recall that $|V_1|\le |V_2|\le |V_3|$, we assert that $|V_1| = |V_2| = 1$. Suppose to the contrary that $|V_2| \ge 2$. Without loss of generality, assume $a_2 \le b_2$. 
	Note that $\sum_{i=1}^3 a_i = \sum_{i=1}^3 b_i = \frac{n}{2}$. We consider two subcases for $|V_1|$.
	If $|V_1|=1$, without loss of generality, assume $a_1=1$ and $b_1=0$. Then $a_1 b_1 = 0$. Since $a_3 = \frac{n}{2} - a_2 - a_1 = \frac{n}{2} - a_2 - 1$ and $b_3 = \frac{n}{2} - b_2 - b_1 = \frac{n}{2} - b_2$, we have $a_1 b_1 + a_3 b_3 = \left(\frac{n}{2} - a_2 - 1\right)\left(\frac{n}{2} - b_2\right)$.
	If $|V_1| = a_1 + b_1 \ge 2$, by Lemma \ref{lem:alg_ineq}, we get $a_1 b_1 + a_3 b_3 \le \left(\frac{n}{2} - a_2 - 1\right)\left(\frac{n}{2} - b_2\right)$.
	
	Then we have
	\begin{align*}
		e(G) &= e(G[V_1]) + e(G[V_2]) + e(G[V_3]) + \sum_{1 \le i < j \le 3} e(V_i, V_j) \\
		&\le a_1b_1+a_2b_2+a_3b_3+2k+1 \\
		&\le a_2b_2+ \left(\frac{n}{2}-a_2-1\right)\left(\frac{n}{2} - b_2\right) + 2k + 1  \\
		&\le \frac{n^2}{4} - \frac{n}{2}(a_2+b_2+1)+2a_2b_2+b_2 + \frac{n-6}{2} && (\text{since } n \ge 4k+8)  \\
		&\le \frac{n(n-3)}{4} - \left(\frac{(2(a_2+b_2)-3)n}{4} - 2a_2b_2 - b_2 + 3\right) \\
		&< \frac{n(n-3)}{4} - \left(\frac{(2(a_2+b_2)-3)(a_2+b_2)}{2} - 2a_2b_2 - b_2 + 3\right) && (\text{since } n > 2(a_2+b_2)) \\
		&= \frac{n(n-3)}{4} - \left(a_2\left(a_2-\frac{3}{2}\right) + b_2\left(b_2-\frac{5}{2}\right) + 3\right) \\
		&< \frac{n(n-3)}{4}. \quad \left(\text{since } a_2\left(a_2-\frac{3}{2}\right)\ge -\frac{9}{16} \text{ and }b_2\left(b_2-\frac{5}{2}\right)\ge -\frac{25}{16}\right)
	\end{align*}
	This contradicts Inequality \eqref{eq:8}. Thus, we must have $|V_1| = |V_2| = 1$.

	Without loss of generality, we assume $V_1=\{v_1\}$ and $V_2=\{v_2\}$. Noting that $d_G(v_1), d_G(v_2)\ge \delta$ and 
	\[
	d_G(v_1) + d_G(v_2)= \sum_{1 \le i < j \le 3} e(V_i, V_j) + e(V_1, V_2)\le 2(k+\frac{\delta-1}{\delta}) + 1=2k+3 - \frac{2}{\delta},
	\]
	we have
	\[
	2\delta\le 2k+3-\frac{2}{\delta} < 2k+3,
	\]
	which contradicts the assumption that $\delta\ge k+2$.
	
	All cases lead to a contradiction. This completes the proof.
\end{proof}

\section*{Declaration of competing interest}
\quad\quad The authors declare that they have no known competing financial interests or personal relationships that could have appeared to influence the work reported in this paper.

\section*{Data availability}
\quad\quad No data was used for the research described in the article.


\begin{thebibliography}{99}

\bibitem{Cai2026} J. Cai, B. Zhou, Eigenvalue conditions implying edge-disjoint spanning trees and a forest with constraints, Discrete Math. 349 (2026) 114710.

\bibitem{Cioaba2022} S.M. Cioab\v{a}, A. Ostuni, D. Park, S. Potluri, T. Wakhare, W. Wong, Extremal graphs for a spectral inequality on edge-disjoint spanning trees, Electron. J. Combin. 29 (2022) 2.56.

\bibitem{Cioaba2012} S.M. Cioab\v{a}, W. Wong, Edge-disjoint spanning trees and eigenvalues of regular graphs, Linear Algebra Appl. 437 (2012) 630--647.

\bibitem{Duan2020} C.X. Duan, L.G. Wang, X.X. Liu, Edge connectivity, packing spanning trees, and eigenvalues of graphs, Linear Multilinear Algebra 68 (2020) 1077--1095.

\bibitem{Fan2023} D.D. Fan, X.F. Gu, H.Q. Lin, Spectral radius and edge-disjoint spanning trees, J. Graph Theory 104 (2023) 697--711.

\bibitem{Fan2025} D.D. Fan, R.M. He, Y.H. Zhao, Distance spectral radius and edge-disjoint spanning trees, Discrete Appl. Math. 376 (2025), 31–40.

\bibitem{Fang2025} X.Q. Fang, D.Q. Yang, An extension of Nash-Williams and Tutte's Theorem, J. Graph Theory 108 (2025) 361--367.

\bibitem{Gu2016} X.F. Gu, H.J. Lai, P. Li, S.M. Yao, Edge-disjoint spanning trees, edge connectivity, and eigenvalues in graphs, J. Graph Theory 81 (2016) 16--29.

\bibitem{Hong2016} Y.M. Hong, X.F. Gu, H.J. Lai, Q.H. Liu, Fractional spanning tree packing, forest covering and eigenvalues, Discrete Appl. Math. 213 (2016) 219--223.

\bibitem{Hu2023} Y. Hu, L.G. Wang, C.X. Duan, Spectral conditions for edge connectivity and spanning tree packing number in (multi-) graphs, Linear Algebra Appl. 664 (2023) 324--348.

\bibitem{Kirchhoff1847} G. Kirchhoff, Über die Auflösung der Gleichungen, auf welche man bei der Untersuchung der linearen Vertheilung galvanischer Ströme geführt wird, Ann. Phys. Chem. 148 (1847) 497--508.

\bibitem{Liu2014} Q.H. Liu, Y.M. Hong, X.F. Gu, H.J. Lai, Note on edge-disjoint spanning trees and eigenvalues, Linear Algebra Appl. 458 (2014) 128--133.

\bibitem{2Liu2014} Q.H. Liu, Y.M. Hong, H.J. Lai, Edge-disjoint spanning trees and eigenvalues, Linear Algebra Appl. 444 (2014) 146--151.

\bibitem{Liu2019} R.F. Liu, H.J. Lai, Y.Z. Tian, Spanning tree packing number and eigenvalues of graphs with given girth, Linear Algebra Appl. 578 (2019) 411--424.

\bibitem{Nash-Williams1961} C.St.J.A. Nash-Williams, Edge-disjoint spanning trees of finite graphs, J. London Math. Soc. 36 (1961) 445--450.

\bibitem{Tutte1961} W.T. Tutte, On the problem of decomposing a graph into $n$ connected factors, J. London Math. Soc. 36 (1961) 221--230.
	
\end{thebibliography}
\end{document}